\documentclass[12pt,twoside]{amsart}
\usepackage{amsmath}
\usepackage{amssymb}
\usepackage{amsthm}
\usepackage{latexsym}
\usepackage{amscd}
\usepackage{xypic}
\xyoption{curve}
\usepackage{ifthen}
\usepackage{hyperref}
\usepackage{graphicx}

\usepackage{xcolor}

\newtheorem{theorem}{Theorem}[section]
\newtheorem{lemma}[theorem]{Lemma}
\newtheorem{proposition}[theorem]{Proposition}
\newtheorem{corollary}[theorem]{Corollary}

\theoremstyle{definition}
\newtheorem{definition}[theorem]{Definition}

\newcommand {\Ann}{\mathrm{Ann}}
\newcommand {\ord}{\mathrm{ord}}
\newcommand {\emdim}{\mathrm{emdim}}

\newcommand {\gr}{\mathrm{gr}}

\newcommand {\sdeg}{\mathrm{sdeg}}
\newcommand {\cdeg}{\mathrm{cdeg}}

\newcommand {\Soc}{\mathrm{Soc}}
\newcommand {\ldf}{\mathrm{ldf}}
\newcommand {\tdf}{\mathrm{tdf}}
\newcommand {\tor}{\mathrm{Tor}}
\newcommand {\Hilb}{\mathcal{H}\kern -0.25ex{\mathit ilb\/}}

\newcommand {\fN}{\mathfrak{N}}
\newcommand {\fM}{\mathfrak{M}}

\newcommand {\bN}{\mathbb{N}}

\def\ga#1{{{\accent"12 #1}}}

\title[Rationality of the Poincar\'e series]{On the rationality of Poincar\'e series\\ of Gorenstein algebras\\ via Macaulay's correspondence}



\author[G. Casnati, J. Jelisiejew, R. Notari]{Gianfranco Casnati, Joachim Jelisiejew, Roberto Notari}
\thanks{The first and third authors are members of GNSAGA group of INdAM. They are supported by the framework of PRIN
2010/11 ``Geometria delle variet\ga a algebriche", cofinanced by MIUR}
\thanks{ The
second author is supported by the project ``Secant varieties, computational complexity, and toric degenerations''
realised within the Homing Plus programme of Foundation for Polish Science, co-financed from European Union, Regional
Development Fund.}
\thanks{This paper is a part of ``Computational complexity,
    generalised Waring type problems and tensor decompositions'' project
    within ``Canaletto'',  the executive program for scientific and
    technological cooperation between Italy and Poland, 2013-2015.}

\begin{document}

\maketitle

\begin{abstract}
Let $A$ be a local Artinian Gorenstein ring with algebraically
closed residue field $A/\fM=k$ of characteristic $0$, and let
$P_A(z) := \sum_{p=0}^{\infty} (\tor_p^A(k,k))z^p $ {be} its
Poicar\'{e} series. We prove that $P_A(z)$ is rational if  either
$\dim_k({\fM^2/\fM^3}) \leq 4 $ and $ \dim_k(A) \leq 16,$ or
there exist $m\leq 4$ and $c$ such that the
Hilbert function $H_A(n)$ of $A$ is equal to $ m$
for $n\in [2,c]$ and equal to $1$ for $n > c$. The
results are obtained thanks to a decomposition of the apolar ideal
$\Ann(F)$ when $F=G+H$ and $G$ and $H$ belong to polynomial rings
in different variables.
\end{abstract}

\section{Introduction and notation}\label{sIntrNot}

Throughout this paper, by ring we mean a Noetherian, associative,
commutative  and unitary ring $A$ with maximal ideal $\frak M$ and
algebraically closed residue field $k:=A/\fM$ of characteristic
$0$.

In \cite{Se} the author asked if the Poincar\'e series of the
local ring $A$, i.e.
$$
P_A(z):=\sum_{p=0}^{\infty}\dim_k(\tor_p^A(k,k))z^p,
$$
is rational. Moreover he also proved its rationality when $A$ is a
regular local ring. Despite many interesting results showing the
rationality of the Poincar\'e series of some rings, in \cite{An}
the author gave an example of an Artinian local algebra $A$ with
transcendental $P_A$. Later on the existence of an Artinian,
Gorenstein, local ring with $\fM^4=0 $   and  transcendental $P_A$
was proved in \cite{Bo}.

Nevertheless,  several results show that large classes of local
rings $A$ have rational Poincar\'e series, e.g. complete
intersections rings (see \cite{Ta}), Gorenstein local rings with
$\dim_k(\fM/\fM^2)\le4$ (see \cite{A--K--M} and \cite{J--K--M}),
Gorenstein local rings with $\dim_k(\fM^2/\fM^3)\le2$ (see
\cite{Sa2}, \cite{E--V3}),   Gorenstein local rings of
multiplicity at most $10$ (see \cite{C--N4}), Gorenstein local
algebras with $\dim_k(\fM^2/\fM^3)=4$ and $\fM^4=0$ (see
\cite{C--E--N--R}).

All the above results are based on the same smart combination of
results on the Poincar\'e series from \cite{A--L} and \cite{G--L}
{first} used in \cite{Sa2} combined with suitable structure
results on Gorenstein rings and algebras. In this last case a
fundamental role has been played by Macaulay's correspondence.

In Section \ref{sMacCor} we give a quick resum\'e of the main
results that we need later on in the paper about Macaulay's
correspondence. In Section \ref{sSimpl} we extend to arbitrary
algebras a very helpful decomposition result already used in a
simplified form in \cite{E--R1} and \cite{C--E--N--R}  for
algebras with $\fM^4=0$. In Section  \ref{sPoinc} we explain how
to relate the rationality of the Poincar\'e series of Gorenstein
algebras with their representation in the setup of Macaulay's
correspondence making use of the aforementioned decomposition
result. Finally, in Section \ref{sAltPoinc} we use such
relationship in order to prove the two following results.

\medbreak \noindent {\bf Theorem A.} {\it Let $A$ be an Artinian,
Gorenstein local $k$--algebra with maximal ideal $\fM$. If there
are integers $m\le4$ and $c\ge1$ such that
$$
\dim_k(\fM^t/\fM^{t+1}) = \left\lbrace\begin{array}{ll}
m\qquad&\text{if $t=2,\dots, c$,}\\
1\qquad&\text{if $t=c+1$,}
 \end{array}\right.
$$
then $P_A$ is rational.}
\medbreak

\noindent {\bf Theorem B.} {\it Let $A$ be an Artinian, Gorenstein
local $k$--algebra with maximal ideal $\fM$. If
$\dim_k(\fM^2/\fM^{3})\le4$ and $\dim_k(A)\le16$, then $P_A$ is
rational.} \medbreak

The above theorems generalize the quoted results on stretched,
almost--stretched and short algebras (see \cite{Sa2},
\cite{E--V3}, \cite{C--N4}, \cite{C--E--N--R}).

\subsection{Notation}
In what follows $k$ is an algebraically closed field of
characteristic $0$. A $k$--algebra is an associative, commutative
and unitary algebra over $k$. For each $N\in\bN$ we set
$S[N]:=k[[x_1,\dots,x_N]]$ and $P[N]:=k[y_1,\dots,y_N]$. We denote
by $S[N]_q$ (resp. $P[N]_q$) the homogeneous component of degree
$q$ of such a graded $k$--algebra, and we set $S[N]_{\le
q}:=\bigoplus_{i=1}^qS[N]_i$ (resp. $P[n]_{\le
q}:=\bigoplus_{i=1}^qP[n]_i$). Finally, we set
$S[n]_+:=(x_1,\dots,x_n)\subseteq S[n]$. The ideal $S[n]_+$ is the unique
maximal ideal of $S[N]$.

A local ring $R$ is Gorenstein if its injective dimension as
$R$--module is finite.

If $\gamma:=(\gamma_1,\dots,\gamma_N)\in{\Bbb N}^{N}$ is a
multi--index, then we set
$t^\gamma:=t_1^{\gamma_1}\dots t_N^{\gamma_N}\in k[t_1,\dots,t_N]$.

For all the other notations and results we refer to \cite{Ha}.


\section{Preliminary results}\label{sMacCor}

In this section we list the main results on algebras we need in
next sections. Let $A$ be a local, Artinian $k$--algebra with
maximal ideal $\fM$. We denote by $H_A$ the Hilbert function of
the graded associated algebra
$$
\gr(A):=\bigoplus_{t=0}^{+\infty}\fM^t/\fM^{t+1}.
$$

We know that
$$
A\cong  S[n] /J
$$
for a suitable ideal $J\subseteq S[n]_+^2\subseteq  S[n]$, where
$n=\emdim(A):=H_A(1)$. Recall that the {\sl socle degree}
$\sdeg(A)$ of $A$ is the greatest integer $s$ such that
$\fM^s\ne0$.

We have an action of $S[n]$ over
$P[n]$ given by partial derivation defined by identifying
$x_i$ with $ \partial/ \partial {y_i}$.  Hence
$$
x^{\alpha}\circ y^{\beta}:=\left\lbrace\begin{array}{ll}
\alpha!{\beta\choose\alpha}y^{\beta-\alpha}\qquad&\text{if $\beta\ge\alpha$,}\\
0\qquad&\text{if $\beta\not\ge\alpha$.}
 \end{array}\right.
$$
Such an action endows  $ P[n] $ with a structure of module over $
S[n] $. If $J\subseteq  S[n] $ is an ideal and  $M\subseteq  P[n]
$ is a $ S[n] $--submodule we set
\begin{gather*}
J^\perp:=\{\ F\in  P[n] \ \vert\ g\circ F=0,\ \forall g\in J\ \},\\
\Ann(M):=\{\ g\in  S[n] \ \vert\ g\circ F=0,\ \forall F\in M\ \}.
\end{gather*}

For the following results see e.g. \cite{Em}, \cite{Ia2} and the
references therein. Macaulay's theory of inverse system is based
on the fact that constructions $J\mapsto J^\perp$ and
$M\mapsto \Ann(M)$ give rise to a {inclusion--reversing} bijection
between ideals $J\subseteq S[n]$ such that $S[n] /J$ is a local
Artinian $k$--algebra and finitely generated $S[n]$--submodules
$M\subseteq P[n]$. In this bijection Gorenstein algebras $A$ with
$\sdeg(A)=s$ correspond to cyclic $ S[n] $--submodules $\langle
F\rangle_{S[n]}\subseteq P[n] $ generated by a polynomial $F$ of
degree $s$. We simply write $\Ann(F)$ instead of $\Ann(\langle
F\rangle_{S[n]})$.

On the one hand, given a $S[n]$--module $M$, we define
$$
\tdf(M)_q:=\frac{{M}\cap P[n]_{\le q}+P[n]_{\le q-1}}{P[n]_{\le q-1}}
$$
where $P[n]_{\le q}:=\bigoplus_{i=0}^qP[n]_i$, and
$\tdf(M):=\bigoplus_{q=0}^\infty\tdf(M)_q$. The module $\tdf(M)$ can be
interpreted as the $S[n]$--submodule of $P[n]$ generated by the
top degree forms of all polynomials in $M$.

On the other hand, for each  $f\in S[n]$, the lowest degree of
monomials appearing with non--zero coefficient in the minimal
representation of $f$ is called {\sl the order of $f$}\/ and it is
denoted by $\ord(f)$. If $f=\sum_{i=\ord(f)}^{\infty}f_i$, $f_i
\in S[n]_i$ then $f_{\ord(f)}$ is called {\sl the lower degree
form of $f$}\/. It will be denoted in what follows with $\ldf(f)$.

If $f\in J$, then $\ord(f)\ge2$. The {\sl lower degree form ideal}
$\ldf(J)$ associated to $J$ is
$$
\ldf(J):=(\ldf(f)\vert f\in J)\subseteq  S[n] .
$$

We have $\ldf(\Ann(M))=\Ann(\tdf(M))$ (see \cite{Em}: see also
\cite{E--R1}, Formulas (2) and (3)) whence
$$
\gr(S[n]/\Ann(M))\cong S[n]/\ldf(\Ann(M))\cong S[n]/\Ann(\tdf(M)).
$$
Thus
\begin{equation}
\label{DerMod}
H_{S[n]/\Ann(M)}(q)=\dim_k(\tdf(M)_q).
\end{equation}
We say that $M$ is {\sl non--degenerate}\/ if
$H_{S[n]/\Ann(M)}(1)=\dim_k(\tdf(M)_1)=n$, i.e. if and only if
the classes of $y_1,\dots,y_n$ are in $\tdf(M)$. If $M=\langle F
\rangle_{S[n]}$, then we write $\tdf(F)$ instead of $\tdf(M)$.

Let $A$ be Gorenstein with $s:=\sdeg(A)$, so that $\Soc(A)=\fM^s\cong k$. In particular $A\cong  S[n] /\Ann(F)$, where $F:=\sum_{i=0}^sF_i$, $F_i\in P[n]_i$. 
For each $h\ge0$ we set $F_{\ge h}:=\sum_{i=h}^sF_i$ (hence $F_s=F_{\ge s}$). We have that $\tdf(F_{\ge h})_i\subseteq\tdf(F)_i$ and equality obviously holds if $i\ge h-1$ (see Lemma 2.1 of \cite{C--N5}).

Trivially, if $s\ge1$, we can always assume that the homogeneous part of $F$ of degree $0$ vanishes, i.e. $F=F_{\ge1}$. Moreover, thanks to Lemma 2.2 of \cite{C--N5} we know that, if $s\ge2$ and $\Ann(F)\subseteq S[n]_+^2$, then we can also assume $F_1=0$, i.e. $F=F_{\ge2}$: we will always make such an assumption in what follows.

We have a filtration with proper ideals (see \cite{Ia2}) of $\gr(A)\cong S[n]/\ldf(\Ann(F))$
$$
C_A(0):=\gr(A)\supset C_A(1)\supseteq C_A(2)\supseteq \dots\supseteq C_A(s-2)\supseteq C_A(s-1):=0.
$$
Via the epimorphism $S[n]\twoheadrightarrow \gr(A)$ we obtain an induced filtration
$$
\widehat{C}_A(0):=S[n]\supset \widehat{C}_A(1)\supseteq \widehat{C}_A(2)\supseteq \dots\supseteq \widehat{C}_A(s-2)\supseteq \widehat{C}_A(s-1):=\ldf(\Ann(F)).
$$

The quotients $Q_A(a):=C_A(a)/C_A(a+1)\cong  \widehat{C}_A(a)/ \widehat{C}_A(a+1)$ are reflexive graded
$\gr(A)$--modules whose Hilbert function is symmetric around  $(s-a)/2$. In general $\gr(A)$ is no more Gorenstein, but the first quotient
\begin{equation}
\label{GorQuot}
G(A):=Q_A(0)\cong S[n] /\Ann(F_s)
\end{equation}
is characterized by the property of being the unique (up to
isomorphism) graded Gorenstein quotient  $k$--algebra of $\gr(A)$
with the same socle degree.  Moreover, the Hilbert function of $A$ satisfies
\begin{equation}
\label{GorDec}
H_A(i)=H_{\gr(A)}(i)=\sum_{a=0}^{s-2}H_{Q_A(a)}(i),\qquad i\ge0.
\end{equation}
Since $H_A(0)=H_{G(A)}(0)=1$, it follows that if $a\ge1$, then $Q_A(a)_0=0$, whence $Q_A(a)_i=0$ when $i\ge s-a$ (see \cite{Ia2}) for the same values of $a$.

Moreover
$$
H_{\gr(A)/C_A(a+1)}(i)=H_{S[n]/\widehat{C}_A(a+1)}(i)=\sum_{\alpha=0}^{a}H_{Q_A(\alpha)}(i),\qquad i\ge0.
$$
We set
$$
f_h:=\sum_{\alpha=0}^{s-h}H_{Q_A(\alpha)}(1)=H_{S[n]/\widehat{C}_A(s-h+1)}(1)=H_{\gr(A)/{C}_A(s-h+1)}(1)
$$
(so that $n=H_A(1)=f_{2}$). 

Finally we introduce the following new invariant.

\begin{definition}
\label{dCapital} Let $A$ be a local, Artinian $k$--algebra with
maximal ideal $\fM$ and $s:=\sdeg(A)$. The {\sl capital degree}\/,
$\cdeg(A)$, of $A$ is defined as the maximum integer $i$, if any,
such that $H_A(i)>1$, $0$ otherwise. If $c=\cdeg(A)$ we also say
that $A$ is a $c$--stretched algebra (for short, stretched if $c\leq 1$).
\end{definition}

By definition $\cdeg(A)\ge0$ and $\cdeg(A)\le \sdeg(A)$: if $A$ is
Gorenstein, then we also have $\cdeg(A)< \sdeg(A)$.

The rationality of the Poincar\'e series $P_A$ of every stretched
ring $A$ is proved in \cite{Sa2}. The proof has been generalized to
rings with $H_A(2)=2$ in \cite{E--V3} and to rings with
$H_A(2)=3$, $H_A(3)=1$ in \cite{C--N4} . The rationality of $P_A$
when $A$ is a $2$--stretched algebra has been studied in
\cite{C--E--N--R} with the restriction $\sdeg(A)=3$.


\section{Decomposition of the apolar ideal}\label{sSimpl}

In the present section we explain how to decompose the ideal
$\Ann(F)$ as the sum of two simpler ideals. Such a decomposition
will be used in the next section in order to reduce the
calculation of the Poincar\'e series of $A$ to the one of a
simpler algebra.

\begin{lemma}
\label{lDecomp} Let $m\le n$, $G\in P[m]$, $H\in
k[y_{m+1},\dots,y_n]$ and $F=G+H$. Let us denote by $\Ann(G)$ and
$\Ann(H)$ the annihilators of $G$ and $H$ inside $S[m]$ and
$k[[x_{m+1},\dots,x_n]]$ respectively. Then
$$
\Ann(F)=\Ann(G)S[n]+\Ann(H)S[n]+(\sigma_G-\sigma_H,x_ix_j)_{1\le i\le m,\ m+1\le j\le n}.
$$
where $\sigma_G\in S[m]$ and $\sigma_H\in k[[x_{m+1},\dots,x_n]]$
are any series of order $\deg(G)$ and $\deg(H)$ such that
$\sigma_G\circ G=\sigma_H\circ H=1$.
\end{lemma}
\begin{proof}
The inclusions $\Ann(G)S[n],\Ann(H)S[n]\subseteq\Ann(F)$ are
completely trivial. Also the inclusion
$(\sigma_G-\sigma_H,x_ix_j)_{1\le i\le m,\ m+1\le j\le
n}\subseteq\Ann(F)$ is easy to check. Thus
$$
\Ann(G)S[n]+\Ann(H)S[n]+(\sigma_G-\sigma_H, x_ix_j)_{1\le i\le m,\
m+1\le j\le n}\subseteq\Ann(F).
$$

Conversely let $p\in \Ann(F)$. Grouping the different monomials in
$p$, we can write a decomposition $p=p_{\le m}+p_{> m}+p_{mix}$,
where $p_{\le m}\in S[m]$, $p_{>m}\in k[[x_{m+1},\dots,x_n]]$ and,
finally, $p_{mix}\in (x_ix_j)_{1\le i\le m,\ m+1\le j\le
n}\subseteq S[n]$.

It is clear that $p_{mix}\in \Ann(G)S[n]+\Ann(H)S[n]+
(\sigma_G-\sigma_H,x_ix_j)_{1\le i\le m,\ m+1\le j\le n}$, hence
it suffices to prove that
$$
p_{\le m}+p_{> m}\in \Ann(G)S[n]+\Ann(H)S[n]+
(\sigma_G-\sigma_H,x_ix_j)_{1\le i\le m,\ m+1\le j\le n}.
$$

To this purpose recall that $0=p\circ F=p_{\le m}\circ G+p_{>
m}\circ H$, by definition. Hence $p_{\le m}\circ G=u=-p_{> m}\circ
H$. Since $p_{\le m}\circ G\in P[m]$ and $p_{> m}\circ H\in
k[y_{m+1},\dots,y_n]$, it follows that $u\in k$. So  $p_{\le
m}-u(\sigma_G-\sigma_H)\in \Ann(G)S[n]$, whence
\begin{align*}
p_{\le m}&\in (\sigma_G-\sigma_H)+\Ann(G)S[n]\subseteq\\
&\subseteq\Ann(G)S[n]+\Ann(H)S[n]+(\sigma_G-\sigma_H,
x_ix_j)_{1\le i\le m,\ m+1\le j\le n}.
\end{align*}
A similar argument shows that
\begin{align*}
p_{> m}&\in (\sigma_G-\sigma_H)+\Ann(H)S[n]\subseteq\\
&\subseteq\Ann(G)S[n]+\Ann(H)S[n]+(\sigma_G-\sigma_H,
x_ix_j)_{1\le i\le m,\ m+1\le j\le n},
\end{align*}
and this concludes the proof.
\end{proof}

Let $F$ be as in the statement above. Then Lemma \ref{lDecomp}
with $G:=\sum_{i=2}^sF_i$ and $H:=\sum_{j=m+1}^ny_j^2$ yield the
following corollary.

\begin{corollary}
\label{cDecomp} Let $m\le n$, $G\in P[m]$ non--degenerate and
$F=G+\sum_{j=m+1}^ny_j^2$. Let us denote by $\Ann(G)$ the
annihilator of $G$ inside $S[m]$. Then
$$
\Ann(F)=\Ann(G)S[n]+(x_j^2-2\sigma,x_ix_j)_{1\le i< j\le n,\ j\ge m+1}
$$
where $\sigma\in S[m]$ has order $\deg(G)$ and $\sigma\circ G=1$.
\end{corollary}
\begin{proof}
It suffices to apply Lemma \ref{lDecomp} taking into account that
$\Ann(H)=(x_j^2-x_{m+1}^2,x_ix_j)_{m+1\le i< j\le n,\ j\ge m+1}$
and that $x_{m+1}^2\circ H=2$.
\end{proof}


\section{Rationality of Poincar\'e series}\label{sPoinc}

We now focus on the Poincar\'e series $P_A(z)$ of the algebra $A$ defined in the introduction: we will generalize some classical results (see \cite{Sa2}, \cite{E--V3},
\cite{C--N4}). Out of the decomposition results proved in the
previous section, the main tools we use are the following ones:
\begin{itemize}
\item for each local Artinian, Gorenstein ring $C$ with
$\emdim(C)\ge2$
\begin{equation}
P_C(z)=\frac{P_{C/\Soc(C)}(z)}{1+z^2P_{C/\Soc(C)}(z)}
\label{PoincSoc}
\end{equation}
(see \cite{A--L}); \item for each local Artinian ring $C$ with
maximal ideal $\fN$, if $c_1,\dots,c_h\in \fN\setminus\fN^2$ are
linearly independent elements of $\Soc(C)$, then
\begin{equation}
P_{C}(z)=\frac{P_{{C/(c_1,\dots,c_h)}}(z)}{1-hzP_{C/(c_1,\dots,c_h)}(z)}
\label{PoincQuot}
\end{equation}
(see \cite{G--L}).
\end{itemize}

Let $A$ be a local, Artinian, Gorenstein, $k$--algebra with
$s=\sdeg(A)$ and $n=H_A(1)$. Assume  $A=S[n]/\Ann(F)$ where
$F=G+\sum_{j=m+1}^ny_j^2\in P[n]$  with $G\in P[m]$. Thanks to
Corollary \ref{cDecomp} we have
$$
\Ann(F)+(\sigma,x_{m+1},\dots,x_n)=\Ann(G)S[n]+(\sigma,x_{m+1},\dots,x_n),
$$
thus
$$
{S[n]\over\Ann(F)+(\sigma,x_{m+1},\dots,x_n)}\cong {S[m]\over\Ann(G)+(\sigma)}.
$$
Trivially $S[m]/\Ann(G)$ is a local, Artinian, Gorenstein,
$k$--algebra.

Since $\Soc(A)$ is generated by the class of $\sigma$, it follows
from formula \eqref{PoincSoc} that
$$
P_A(z)=\frac{P_{S[n]/\Ann(F)+(\sigma)}(z)}{1+z^2P_{S[n]/\Ann(F)+(\sigma)}(z)}.
$$
Notice that $x_ix_j\in\Ann(F)+(\sigma)$, $i=1,\dots,n$,
$j=m+1,\dots,n$, $i\le j$. In particular $x_{m+1},\dots,x_n\in
\Soc(S[n]/\Ann(F)+(\sigma))$. It follows from formula
\eqref{PoincQuot} that
$$
P_{S[n]/\Ann(F)+(\sigma)}(z)=\frac{P_{{S[n]/\Ann(F)+(\sigma,
x_{m+1},\dots,x_n)}}(z)}{1-(n-m)zP_{S[n]/\Ann(F)+(\sigma,
x_{m+1},\dots,x_n)}(z)}.
$$
The inverse formula of \eqref{PoincSoc} finally yields
$$
P_{S[n]/\Ann(F)+(\sigma,x_{m+1},\dots,x_n)}=
P_{S[m]/\Ann(G)+(\sigma)}(z)=\frac{P_{S[m]/\Ann(G)}(z)}{1-z^2P_{S[m]/\Ann(G)}(z)}.
$$

Combining the above equalities we finally obtain the following

\begin{proposition}
\label{pRelation} Let $G\in P[m]$, $F:=G+\sum_{j=m+1}^ny_j^2$ and
define $A:=S[n]/\Ann(F)$ and $B:=S[m]/\Ann(G)$. Then
$$
P_A(z)=\frac{P_{{B}}(z)}{1-(H_A(1)-H_A(2))zP_{B}(z)} .
$$
\end{proposition}

A first immediate consequence of the above Proposition is the
following corollary.

\begin{corollary}
\label{cRatGeneral} Let $G\in P[m]$, $F:=G+\sum_{j=m+1}^ny_j^2$
and define $A:=S[n]/\Ann(F)$ and $B:=S[m]/\Ann(G)$. The series $P_B(z)$ is
rational if and only if the same is true for $P_A(z)$.
\end{corollary}

Now assume that $m\le4$. Since  the Poincar\'e series of each
local Artinian, Gorenstein ring with embedding dimension at most
four is rational (see \cite{Se}, \cite{Ta}, \cite{Wi},
\cite{J--K--M}) we also obtain the following corollary.

\begin{corollary}
\label{cRatUpToFour} Let $G\in P[4]$, $F:=G+\sum_{j=5}^ny_j^2$ and
define $A:=S[n]/\Ann(F)$. Then $P_A(z)$ is rational.
\end{corollary}

Let $A$ be a local, Artinian, Gorenstein $k$--algebra with
$n:=H_A(1)$. 

\begin{corollary}
\label{cRatF} Let $A$ be a local, Artinian, Gorenstein
$k$--algebra such that $f_3\le4$. Then $P_A(z)$ is rational.
\end{corollary}
\begin{proof}
If $s:=\sdeg(A)$, then
$$
A\cong S[n]/\Ann(F)
$$
where $F:=\sum_{i=2}^sF_i+\sum_{j=f_3+1}^ny_j^2$, $F_i\in
P[f_i]_i$, $i\ge3$ and $F_2\in P[f_3]_2$ (see \cite{C--N5}, Remark 4.2). Thus
the statement follows from Corollary \ref{cRatUpToFour}.
\end{proof}


\section{Examples of algebras with rational Poincar\'e series}

\label{sAltPoinc} In this section we give some examples of local,
Artinian, Gorenstein $k$--algebras $A$ with rational $P_A$ using
the results proved in the previous section.

We start with the following Lemma generalizing a result in \cite{St}.

\begin{lemma}
\label{lNotExist} Let $A$ be a local, Artinian, Gorenstein,
$3$--stretched $k$--algebra. If $H_A(3)\le5$, then
$\sum_{a=0}^{s-4}H_{Q_A(a)}(2)\ge H_A(3)$.
\end{lemma}
\begin{proof}
We set $m:=H_A(3)$ and $p:=\sum_{a=0}^{s-4}H_{Q_A(a)}(2)$. We have to show that $p\ge m$: assume $p\le m-1$. 

If $s=4$, then $H_{Q_A(0)}=\sum_{a=0}^{s-4}H_{Q_A(a)}=(1,m,p,m,1)$.
If $s\ge 5$, then we have
$$
H_{Q_A(a)}=\left\lbrace\begin{array}{ll}
(1,1,1,1,1,\dots,1)\qquad&\text{if $a=0$,}\\
(0,0,0,0,0,\dots,0)\qquad&\text{if $a=1,\dots,s-5$,}\\
(0,m-1,p-1,m-1,0,\dots,0)\qquad&\text{if $a=s-4$.}
 \end{array}\right.
$$
In particular $\sum_{a=0}^{s-4}H_{Q_A(a)}=H_{Q_A(0)}+H_{Q_A(s-4)}$. Notice that $f_{4}=m$.

Macaulay's growth theorem (see \cite{B--H},
Theorem 4.2.10) and the restriction $m\le5$ imply that $3\le
p=m-1$ necessarily. Thus we can restrict our attention to the two
cases $p=3,4$. We examine the second case, the first one being
analogous.

Let  $n:=H_A(1)$, take a polynomial $F:=y_1^s+F_4+F_3+F_2$,
$F_i\in P[f_i]_i$, $x_1^3\circ F_4=0$ such that $A\cong
S[n]/\Ann(F)$ (see Remark 4.2 of \cite{C--N5}) and set
$B:=S[n]/\Ann(F_{\ge 4})$.

We first check that $H_B=\sum_{a=0}^{s-4}H_{Q_A(a)}=(1,5,4,5,1,\dots,1)$. On the one hand, Lemma 1.10 of \cite{Ia2} implies that $\widehat{C}_A(a)=\widehat{C}_B(a)$, $a\le s-3$, whence 
$$
H_B(1)\ge \sum_{a=0}^{s-4}H_{Q_B(a)}(1)=\sum_{a=0}^{s-4}H_{Q_A(a)}(1)=5.
$$
On the other hand, $F_{\ge4}\in P[f_4]=P[5]$, whence $5=H_B(1)\le5$. It follows that equality holds, thus $H_{Q_B(s-2)}(1)=H_{Q_B(s-3)}(1)=0$. By symmetry we finally obtain $H_{Q_B(s-2)}=H_{Q_B(s-3)}=0$. This last vanishing completes the proof of the equality $H_B=\sum_{a=0}^{s-4}H_{Q_A(a)}=(1,5,4,5,1,\dots,1)$.

Let $I\subseteq k[x_1,\dots,x_n]\subseteq S[n]$ be the ideal
generated by the forms of degree at most $2$ inside
$\Ann(\tdf(F_{\ge 4}))=\ldf(\Ann(F_{\ge4}))$. We obviously have $x_6,\dots,x_n\in I$, because $F_{\ge4}\in P[5]$. Denote by $I^{sat}$ the saturation of $I$
and set $R:=k[x_1,\dots,x_n]/I$, $R^{sat}:=k[x_1,\dots,x_n]/I^{sat}$. Due to the definition of $I$ we know that $H_R(t)\ge H_B(t)$ for each $t\ge0$, and equality holds true for $t\le2$.
Moreover, we know that 
$$
H_B(2)^{\langle2\rangle}=H_B(3)\le H_R(3)\le H_R(2)^{\langle2\rangle}=H_B(2)^{\langle2\rangle},
$$
hence
$$
H_R(3)={4\choose3}+{2\choose2}=H_R(2)^{\langle2\rangle}.
$$
Gotzmann Persistence Theorem (see \cite{B--H}, Theorem 4.3.3) implies that
$$
H_{R}(t)={t+1\choose t}+{t-1\choose t-1}=t+2,\qquad t\ge2.
$$
We infer $H_{R^{sat}}(t)=t+2$, $t\gg0$.

When saturating, the ideal can only increase its size in each
degree, hence $H_{R^{sat}}(t)\le H_{R}(t)$ for each
$t\ge0$. Again Macaulay's bound thus forces
$H_{R^{sat}}(t)=H_{R}(t)=t+2$ for $t\ge2$. In particular
the components $I_t$ and $I_t^{sat}$ of degree $t\ge2$ of $I$ and
$I^{sat}$ coincide.

Since $H_{R^{sat}}$ is non--decreasing, it follows that
$$
H_{R^{sat}}(1)\le H_{R^{sat}}(2)=4<5=H_B(1)=H_{R}(1).
$$
In particular there exists a linear form $\ell\in I^{sat}\setminus
I$. The equality $I_2=I_2^{sat}$ forces $\ell x_j\in I_2\subseteq
\Ann(\tdf(F_{\ge 4}))$, $j=1,\dots,n$. Since $x_6,\dots,x_n\in I$,
it follows that we can assume $\ell\in S[5]\subseteq S[n]$.
Moreover we also know that $y_1^s\in \tdf(F_{\ge 4})$, hence
$\ell$ cannot be a multiple of $x_1$. In particular we can change
linearly coordinates in such a way that $\ell=x_5$.

If $j\ge2$, then $x_j\circ F_{\ge4}=x_j\circ F_4$, thus the
condition $x_jx_5\in I_2\subseteq \Ann(\tdf(F_{\ge 4}))$,
$j=2,\dots,5$, and $x_1^3\circ F_4=0$ imply that $x_5\circ
F_{4}=0$. Such a vanishing contradicts the linear independence of the derivatives
$$
x_2\circ F_{\ge4},\quad x_3\circ F_{\ge4},\quad x_4\circ F_{\ge4},
\quad x_5\circ F_{\ge4}.
$$
Indeed $5=H_B(1)=\dim_k(\tdf(F_{\ge4})_1)$ and $x_j\circ F_{\ge4}=0$, $j\ge6$.
\end{proof}

Using the results proved in the previous section and the Lemma
above we are able to handle the first example of this section,
proving the following theorem generalizing Corollary 2.2 of \cite{C--E--N--R}. 
\begin{theorem}
\label{tRat3Str} 
Let $A$ be a local, Artinian, Gorenstein
$k$--algebra with $H_A(2)\le4$ and $\cdeg(A)\le3$. Then $P_A$ is
rational.
\end{theorem}
\begin{proof}
Let us examine the case $\cdeg(A)=3$, the other ones being similar. Lemma \ref{lNotExist} yields
\begin{equation}\label{e3StrIneq}
    H_A(2)\ge \sum_{a=0}^{s-4}H_{Q(a)}(2)\ge H_A(3).
\end{equation}
If $\sdeg(A) \geq 5$, then Decomposition \eqref{GorDec} is 
$$ 
(1,1, \dots, 1) + (0, a_1, a_2, a_1, 0) + (0, b_1, b_1, 0) + (0, c_1,0)
$$ 
for some integers $a_1, a_2, b_1, c_1$. Inequality
\eqref{e3StrIneq} is equivalent to $a_1\leq a_2$. We know that
$H_A(2) = a_2 + b_1 + 1\leq 4$, so $f_4 = a_1 + b_1 + 1\leq  4$
and the argument follows from Corollary \ref{cRatF}. In the case
$\sdeg(A) = 4$, the decomposition \eqref{GorDec} changes, but the
argument stays the same.
\end{proof}

Now we skip the condition $\cdeg(A)=3$ but we impose a restriction on the shape of $H_A$. The following theorem generalizes a well--known result
proved when either $m=1,2$ (see \cite{Sa2} and \cite{E--V3}
respectively) or $m\le 4$ and $s=3$ (see again \cite{C--E--N--R}).

\begin{theorem}
\label{tColumn} Let $A$ be a local, Artinian, Gorenstein
$k$--algebra such that $H_A(i)=m$, $2\le i\le \cdeg(A)$. If $m\le
4$, then $P_A$ is rational.
\end{theorem}
\begin{proof}
Let $c:=\cdeg(A)$, $n:=H_A(1)$, take a polynomial $F:=y_1^s+F_{c+1}+\dots$,
$F_{c+1}\in P[f_{c+1}]_{c+1}=P[m]_{c+1}$ such that $A\cong
S[n]/\Ann(F)$ (see Remark 4.2 of \cite{C--N5}) and set $B:=S[n]/\Ann(F_{\ge c+1})$ so that $Q_A(a)= Q_B(a)$ for $a\le s-c-1$ (again by Lemma 1.10 of \cite{Ia2}). In particular $H_B(c)=m$, thus Decomposition \eqref{GorDec} implies $H_B(1)\ge m$. Since we know that $F_{\ge c+1}\in P[m]$, it follows that $H_B(1)\le m$, hence equality must hold.

As in the proof of the previous lemma one immediately checks that
either $s=c+1$, and $H_{Q_A(0)}=(1,m,\dots,m,1)$, or $s\ge c+2$, and
$$
H_{Q_A(a)}\left\lbrace\begin{array}{ll}
(1,1,\dots,1,1,\dots,1)\qquad&\text{if $a=0$,}\\
(0,0,\dots,0,0,\dots,0)\qquad&\text{if $a=1,\dots,s-c-2$,}\\
(0,m-1,\dots,m-1,0,\dots,0)\qquad&\text{if $a=s-c-1$.}
 \end{array}\right.
$$

Assume that $H_{B}(i)\le m-1\le3$ for some
$i=2,\dots,c-1$. Let $i_0$ be the maximal of such $i$'s. We know
that there are $k(i_0)>k(i_0-1)>k(i_0-2)>\dots$ such that
$$
H_{B}(i_0)={k(i_0)\choose i_0}+{k(i_0-1)\choose i_0-1}
+{k(i_0-2)\choose i_0-2}+\dots\le m-1\le3.
$$
If $i_0\ge3$, it would follow $k(i_0)\le i_0$, thus Macaulay's
bound implies
\begin{align*}
H_{B}(i_0+1)&\le H_{B}(i_0)^{\langle i_0\rangle}=\\
&={k(i_0)+1\choose i_0+1}+{k(i_0-1)+1\choose i_0}+{k(i_0-2)+1\choose i_0-1}+\dots=\\
&=H_{B}(i_0)\le m-1,
\end{align*}
a contradiction. We conclude that $i_0=2$. 

Due to the symmetry of
$H_{Q_B(s-c-1)}$ we deduce that $c=3$. If $H_{Q_B(s-3)}(2)=q$, the symmetry of $H_{Q_B(s-3)}$ implies $H_{Q_B(s-3)}(1)=q$, hence Decomposition \eqref{GorDec} implies
$$
m=H_B(1)=\sum_{a=0}^{s-2}H_{Q_B(a)}(1)= m+q+H_{Q_B(s-2)}(1).
$$
It follows that $q=H_{Q_B(s-2)}(1)=0$, whence $H_{B}=(1,m,p,m,1,\dots,1)$ where $p\le m-1$ which cannot occur by Lemma
\ref{lNotExist}.

We conclude that $H_{Q_A(s-c-1)}(i)=H_{Q_B(s-c-1)}(i)=m-1$ for each $i=2,\dots,c$, then the hypothesis on $H_A(i)$ and Decomposition \eqref{GorDec} yield
$$
H_{Q_A(a)}=\left\lbrace\begin{array}{ll}
(0,0,0,\dots,0,0,\dots,0)\qquad&\text{if $a=s-c,\dots,s-3$,}\\
(0,n-m,0,\dots,0,0,\dots,0)\qquad&\text{if $a=s-2$,}
 \end{array}\right.
$$
whence $f_3=\sum_{a=1}^{s-3}H_{Q(a)}(1)=m\le4$.
\end{proof}

As third example we skip the condition on the shape of $H_A$ but
we put a limit on $\dim_k(A)$, slightly extending the result
proved in \cite{C--N5}.

\def\MaxDim{16}%
\begin{theorem}
\label{cRatUpTo16}
Let $A$ be a local, Artinian, Gorenstein $k$--algebra with $\dim_k(A)\le
\MaxDim$ and $H_A(2)\le4$. Then $P_A$ is rational.
\end{theorem}
\begin{proof}
Thanks to \cite{J--K--M} we can restrict our attention to algebras
$A$ with $H_A(1)\ge5$.

The rationality of the Poincar\'e series of stretched algebras is
proved in \cite{Sa2}. For almost stretched algebras see
\cite{E--V3}. For the case of algebras $A$ with $\sdeg(A)=3$ and
$H_A(2)\le4$ see \cite{C--E--N--R}. Finally the case $H_A(i)=m$,
$2\le i\le \cdeg(A)$ with $m\le 4$ is covered by Theorem
\ref{tColumn} above.

There are several cases which are not covered by the aforementioned results.
In each of these cases one can check that the condition $f_3\le 4$ of
Corollary \ref{cRatF} is fulfilled.
We know that necessarily $H_A(2)\ge3$, otherwise $A$ is almost stretched by
Macaulay's bound. The restriction $H_A(2)\le4$ implies $H_A(3)\le5$
again by Macaulay's bound.

Theorem \ref{tRat3Str} deals with the case $\sdeg(A) = 4$. Let us analyze
the case $\sdeg(A) = 5$ and $\dim_k A \leq \MaxDim$. The decomposition is
\[
(1, a_1, a_2, a_2, a_1, 1) + (0, b_1, b_2, b_1, 0) + (0, c_1, c_1, 0) + (0,
d_1, 0)
\]
for some integers $a_1, a_2, b_1, b_2, c_1, d_1$. If $a_1 = 1$
then the algebra is $3$-stretched, so we may suppose $a_1 \geq 2$.
We know that $H_A(2) = a_2 + b_2 + c_1 \leq 4$ and we would like
to prove $a_1 + b_1 + c_1 \leq 4$. Suppose $a_1 + b_1 + c_1 \geq
5$, then the inequality on the dimension of $A$ shows that $2\cdot
a_2 + b_2 \leq 4$, in particular $a_2 \leq 2$ and from Macaulay's
bound it follows that $a_1 = a_2 = 2$. It follows that $b_2 = 0$
and once again from Macaulay's bound $b_1 = 0$. This forces $a_1 +
b_1 + c_1 = 2 + c_1 = a_2 + b_2 + c_1 \leq 4$, a contradiction.

Let us now suppose that $\sdeg(A) = 6$. Look at the first row of the
symmetric decomposition \eqref{GorDec}: $(1, a_1, a_2, a_3, a_2, a_1, 1)$.
\begin{itemize}
    \item If $a_1 \geq 3$, then $a_2, a_3 \geq 3$ and the sum of the row is at least
$17$.
\item If $a_1 = 2$ then $a_2 = a_3 = 2$ and the sum of the row is $12$.
If we suppose that $f_3 \geq 5$, then the sum of the first column of the
remaining part of the decomposition will be at least three, so the sum of
whole remaining part will be at least $2\cdot 3 = 6$ and the dimension will
be at least $12 + 6 > \MaxDim$.
    \item Suppose $a_1 = 1$ and look at the second row $(0, b_1, b_2, b_2, b_1, 0)$.
    If $b_1 = 0$ then the algebra is
        $3$-stretched so the result follows from Theorem \ref{tRat3Str}.
        From $H_A(2) \leq 4$ it follows that $b_2 \leq 3$. If $b_2 = 3$, then
        $b_1 \geq 2$ so the dimension is at least $7 + 10 > \MaxDim$.
        If $b_2 \leq 2$ then $b_1 \leq b_2$ from Macaulay's bound. Hence,
        the same argument as before applies.
\end{itemize}

Let us finally suppose that $\sdeg(A) \geq 7$. Take the first row,
beginning with $(1, a_1, a_2,\dots)$. If $a_1 \geq 3$ then its sum
is at least $3\cdot \sdeg(A) - 1 > \MaxDim$. If $a_1 = 2$, the sum
of this row is $2\cdot \sdeg(A) \geq 14$. Then one can argue as in
the case $\sdeg(A) = 6,\ a_1 = 2$. A similar reasoning shows that
when $a_1 = 1$ the algebra has decomposition $(1, 1, \dots, 1) +
(0, 4, 4, 0)$ and so $H_A(2) \geq 5$.
\end{proof}


\bigskip
\noindent
Gianfranco Casnati,\\
Dipartimento di Scienze Matematiche,  Politecnico di Torino, \\
corso Duca degli Abruzzi 24, 10129 Torino, Italy \\
e-mail: {\tt gianfranco.casnati@polito.it}

\bigskip
\noindent
Joachim Jelisiejew,\\
Faculty of Mathematics, Informatics, and Mechanics, University of Warsaw,\\
Banacha 2, 02-097 Warsaw, Poland\\
\texttt{jj277546@students.mimuw.edu.pl}

\bigskip
\noindent
Roberto Notari, \\
Dipartimento di Matematica, Politecnico di Milano,\\
via Bonardi 9, 20133 Milano, Italy\\
e-mail: {\tt roberto.notari@polimi.it}

\enddocument